\documentclass[10pt,a4paper]{article}
\usepackage[utf8]{inputenc}
\usepackage{amsthm}
\usepackage{amsmath}
\usepackage{amsfonts}
\usepackage{amssymb}
\usepackage{verbatim}
\usepackage{enumitem}
\usepackage{cite}
\newtheorem{theorem}{Theorem}[section]
\newtheorem{lem}[theorem]{Lemma}
\newtheorem{prop}[theorem]{Proposition}
\newtheorem{cor}[theorem]{Corollary} 
\theoremstyle{definition}
\newtheorem*{exmp}{Example}
\newtheorem*{remark}{Remark}
\usepackage{tikz-cd}
\usepackage[margin=1in]{geometry}
\usepackage{microtype}
\usepackage{indentfirst}
\usepackage[small]{titlesec}
\usepackage{ytableau}
\date{}

\author{Patryk Jaśniewski}
\title{On homological properties of strict polynomial functors of \mbox{degree $p$}}
\begin{document}
\maketitle
\begin{abstract}
We study the homological algebra in the category $\mathcal{P}_p$ of strict polynomial functors of degree $p$ over a field of positive characteristic $p$. We determine the decomposition matrix of our category and we calculate the Ext-groups between functors important from the point of view of representation theory. Our results include computations of the Ext-algebras of simple functors and Schur functors. We observe that the category $\mathcal{P}_p$ has a Kazhdan-Lusztig theory and we show that the DG algebras computing the Ext-algebras for simple functors and Schur functors are formal. These last results allow one to  describe the bounded derived category of $\mathcal{P}_p$ as derived categories of certain explicitly described graded algebras. We also generalize our results to all blocks of $p$-weight $1$ in $\mathcal{P}_e$ for $e>p.$\newline

\noindent {\it Keywords}: block, Ext-group, formality, strict polynomial functor
\end{abstract}
\section*{Introduction}
The category $\mathcal{P}$ of strict polynomial functors over a field $k$ was introduced in \cite{fs} by Friendlander and Suslin. Intuitively, one can think about a strict polynomial functor as a functor from the category of finite dimensional vector spaces to itself, which acts on $\operatorname{Hom}$-spaces via polynomial maps. We consider the case of a field $k$ of characteristic $p > 0$. It is well-known that $\mathcal{P}$ is an abelian category.

The category $\mathcal{P}$ is the direct sum of its full subcategories $\mathcal{P}_e$ of homogeneous strict polynomial functors of degree $e.$ It is well-known that $\mathcal{P}_e$ is a highest weight category with the set of isomorphism classes of simples indexed by the poset of Young diagrams of weight $e$. The category $\mathcal{P}_e$ is an important tool in representation theory. Indeed, it is known that $\mathcal{P}_e\simeq S(n,e)-\text{mod}^\text{fin}$ where $S(n,e)-\text{mod}^\text{fin}$ is the category of finitely generated modules over the Schur algebra $S(n,e)$ with $n\geq e.$ Therefore, there is a relationship between $\mathcal{P}_e$ and the category of polynomial representations of degree $e$ of the general linear group $\operatorname{GL}_n(k)$ and between $\mathcal{P}_e$ and the category of representations of the symmetric group $\Sigma_e$ as well. It turns out that homological computations are much simpler in the category $\mathcal{P}_e$ than in the category $S(n,e)-\text{mod}^\text{fin}.$

In this paper we investigate the category $\mathcal{P}_p,$ because it is the first non-trivial case. Indeed, it is known that $\mathcal{P}_e\simeq k[\Sigma_e]-\text{mod}$ for $e<p$ and $k[\Sigma_e]$-mod is a semisimple category by the Maschke theorem. The case $e=p$ is also the starting point for the attempt of understanding the categories $\mathcal{P}_e$ for higher $e$. As we will see, some important structures organising the homological algebra in $\mathcal{P}_e$ like the Koszul and de Rham complexes or formality phenomena investigated in \cite{mc_adv}, \cite{mc_imrn} already play an important role for $e=p$. We point out here that we make use, in an essential way, of strict polynomial functors in the proofs several times. For instance, we observe in Proposition \ref{injective_hull} that the functor $S^{p-i}\otimes \Lambda^i$ is injective, which is fairly straightforward in $\mathcal{P}_p$. An explanation of that in terms of $\operatorname{GL}_n$-representations seems to be significantly more complex. We also use the sum-diagonal adjunction (see, e.g., \cite{ffss}), which is specific to functor categories and is invisible at the level of $\operatorname{GL}_n$-modules.

Let us briefly summarise the results of this paper. First of all, let us observe that the non-trivial homological algebra in the category $\mathcal{P}_p$ appears only in the block of $\mathcal{P}_p$ corresponding to the Young diagrams of shape of hooks. We  denote this block by $\mathcal{P}_p^\varnothing.$

In Section 1, after recalling some standard tools and constructions, we determine injective hulls and projective covers of simples in Proposition \ref{injective_hull} and we compute the decomposition matrix of $\mathcal{P}_p^\varnothing$ in Corollary \ref{decomp_matrix}.

In Section 2 computations of $\operatorname{Ext}^q(F_i,S_j)$, $\operatorname{Ext}^q(W_i,F_j)$, $\operatorname{Ext}^q(F_i,F_j)$, $\operatorname{Ext}^q(S_i,S_j)$, $\operatorname{Ext}^q(W_i,W_j)$, $\operatorname{Ext}^q(S_i,F_j)$, $\operatorname{Ext}^q(F_i,W_j)$ and $\operatorname{Ext}^q(S_i,W_j)$, where $S_i$ (resp. $W_i$, $F_i$) is the costandard (resp. standard, simple) functor corresponding to the hook $(i+1,1^{p-i-1})$, are performed for any $0\leq i,j \leq p-1.$ Computation of $\operatorname{Ext}^q(F_i,S_j)$ gives us  Corollary \ref{kltheory}, which states that the category $\mathcal{P}_p$ has a Kazhdan-Lusztig theory.

Restricting our attention to the block $\mathcal{P}_p^\varnothing,$ we describe, in Theorem \ref{quasiiso_ext_schurs}, the structure of the Yoneda algebra of costandard functors and we prove that the endomorphism algebra whose cohomology algebra is the Yoneda algebra of costandard functors is a formal DG algebra. This theorem implies Corollary \ref{derivedequiv_schurs}, which establishes an equivalence of triangulated categories $\mathcal{D}^b\mathcal{P}_p^\varnothing \simeq \mathcal{D}^b(\operatorname{Ext}^*(S,S)-\text{mod}^\text{gr})$ where $S=\bigoplus_{i=0}^{p-1} S_i.$ We point out that this ``full formality'' result for $\mathcal{P}_p$ seems to underlie certain ``partial formality'' phenomena observed in $\mathcal{P}_e$ for higher $e$ in \cite{mc_adv}, \cite{mc_imrn}. We hope to further explore this connection in a future work.

In a similar fashion, under the restriction to the block $\mathcal{P}_p^\varnothing$, we obtain, in Theorem \ref{quaiiso_ext_simples}, the structure of the Yoneda algebra of simple functors and the DG formality of the endomorphism algebra whose cohomology algebra is the Yoneda algebra of simple functors. We also show that there is a graded algebra isomorphism $\operatorname{Ext}^*(F_i,F_i) \simeq K[x]/(x^{p-i})$ for $0\leq i \leq p-1$ and $x$ of degree $2$. Theorem \ref{quaiiso_ext_simples} implies Corollary \ref{derivedequiv_simples}, which establishes an equivalence of triangulated categories $\mathcal{D}^b\mathcal{P}_p^\varnothing \simeq \mathcal{D}^b(\operatorname{Ext}^*(F,F)-\text{mod}^\text{gr})$ where $F=\bigoplus_{i=0}^{p-1} F_i.$ By Corollaries \ref{derivedequiv_schurs} and \ref{derivedequiv_simples} we obtain Corollary \ref{derivedequiv_schursimple} telling us that $\mathcal{D}^b(\operatorname{Ext}^*(S,S)-\text{mod}^\text{gr}) \simeq \mathcal{D}^b(\operatorname{Ext}^*(F,F)-\text{mod}^\text{gr})$ as triangulated categories. This result may be interesting for its own, since the involved graded algebras, explicitly described in Theorems \ref{quasiiso_ext_schurs} and \ref{quaiiso_ext_simples}, look at first glance quite different.

In the last section we show that the category $\mathcal{P}_p^\varnothing$ is equivalent to the block of $\mathcal{P}_{|\lambda|+p}$ corresponding to a given $p$-core $\lambda$. It provides a generalization of the results obtained in the previous sections to the case of higher degrees. In particular, this result allows one to completely describe the category $\mathcal{P}_e$ for $p < e < 2p.$ 
\section{Preliminaries}
In the first part of this section we recall definitions and facts used in this paper. One can find proofs of these facts in, e.g., \cite{abw}, \cite{donkin}, \cite{fls}, \cite{fs}, \cite{jantzen}, \cite{krause}. In the second part of the section we introduce exact sequences, which will be crucial in the next sections. We also describe injective hulls and projective covers of simples in the category $\mathcal{P}_p$ and the decomposition matrix of $\mathcal{P}_p$.

Let $\mathcal{P}_e$ be the category of homogeneous strict polynomial functors of degree $e$ over a field $k$ of positive characteristic $p$. It is well-known that $\mathcal{P}_e$ is a highest weight category in the sense of \cite{cps} with poset $P$ of Young diagrams of weight $e$ with the reversed dominance order. We now recall the definition of the poset $P$ and the construction of standard and costandard objects of this category called in this paper, respectively, Weyl and Schur functors.

Fix an integer $e>0.$ Let $P$ be the set of partitions of $e,$ i.e. of non-increasing sequences of natural numbers $\lambda=(\lambda_1,\ldots,\lambda_n)$ satisfying the condition $\sum_{i=1}^n \lambda_i=e.$ These sequences can be graphically visualized with Young diagrams of weight $e$ -- for a partition $\lambda=(\lambda_1,\ldots,\lambda_n)$ the $i$-th row has $\lambda_i$ boxes. A partial order on the set $P$ is the reversed dominance order: $\lambda=(\lambda_1,\ldots,\lambda_n)$ dominates over $\mu=(\mu_1,\ldots,\mu_m)$ if and only if $\sum_{j\leq i} \lambda_j \leq \sum_{j\leq i} \mu_j$ for all $1\leq i \leq \max\{m,n\}.$ We are now ready to describe the structure of $\mathcal{P}_e$ as a highest weight category with poset $P$ defined above. We denote the $e$-th tensor (resp. symmetric, divided, exterior) power by $I^{\otimes e}$ (resp. $S^e$, $\Gamma^e$, $\Lambda^e$). For a given diagram $\lambda=(\lambda_1,\ldots,\lambda_n)$ let $\Lambda^\lambda=\Lambda^{\lambda_1}\otimes \ldots \otimes \Lambda^{\lambda_n}, \ S^\lambda=S^{\lambda_1}\otimes \ldots \otimes S^{\lambda_n}, \ c_\lambda=c_{\lambda_1}\otimes \ldots \otimes c_{\lambda_n} : \Lambda^\lambda \to I^{\otimes e}$ and $m_\lambda=m_{\lambda_1}\otimes \ldots \otimes m_{\lambda_n}:I^{\otimes e}\to S^\lambda,$ where $c_r:\Lambda^r \to I^{\otimes r}$ is the natural monomorphism and $m_r:I^{\otimes r}\to S^r$ is the natural epimorphism. Then the Schur functor $S_\lambda$ is defined as $S_\lambda=\operatorname{im}(m_{\widetilde{\lambda}} \circ c_\lambda)$ where $\widetilde{\lambda}$ is the conjugate of $\lambda.$ The category $\mathcal{P}_e$ has the Kuhn duality being the contravariant exact functor $(-)^\#:(\mathcal{P}_e)^{\text{op}} \to \mathcal{P}_e$ defined as follows: $G^\#(V)=G(V^*)^*,$ where $G\in \mathcal{P}_e$ and $V^*$ is the dual space to $V$. Then we define the Weyl functor $W_\lambda$ as the Kuhn dual of $S_\lambda,$ i.e. $W_\lambda=(S_\lambda)^\#.$ By the definition of $S_\lambda$ the functor $W_\lambda$ can be directly expressed as the composition of the Kuhn dual maps to those given above $\Gamma^{\widetilde{\lambda}}\to I^{\otimes e} \to \Lambda^\lambda.$ Finally, the simple functor $F_\lambda$ is defined as the image of the composition of the obvious maps $W_\lambda \hookrightarrow \Lambda^\lambda \to S_\lambda.$ As examples, it is easy to derive from these definitions that $S_{(1)^e}=S^e$, $W_{(1)^e}=\Gamma^e$, $F_{(1)^e}=I^{(1)}$ and $S_{(e)}=W_{(e)}=F_{(e)}=\Lambda^e$, where $I^{(1)}$ is the Frobenius twist functor.

We now recall the definitions of the Koszul and de Rham complexes discussed in \cite{fs}. For a given $e>0$ and $0\leq i \leq e$ set $\Omega^i_e=S^{e-i}\otimes \Lambda^i.$ We refer to $i$ as the cohomological degree. The Koszul differential $\kappa_i:\Omega^i_e \to \Omega^{i-1}_e$ and the de Rham differential $d_i:\Omega^i_e \to \Omega^{i+1}_e$ are defined by using the coproduct and product operations in the symmetric and exterior algebras:
\begin{equation*}
\begin{split}
\kappa_i:S^{e-i}(V)\otimes \Lambda^i(V) \to S^{e-i}(V)\otimes V \otimes \Lambda^{i-1}(V)\to S^{e-i+1}(V)\otimes \Lambda^{i-1}(V), \\ d_j:S^{e-i}(V)\otimes \Lambda^i(V) \to S^{e-i-1}(V)\otimes V \otimes \Lambda^{i}(V)\to S^{e-i-1}(V)\otimes \Lambda^{i+1}(V).
\end{split}
\end{equation*}
The complex $(\Omega^\bullet_e,\kappa)$ is called the Koszul complex and it is well-known that this complex is acyclic. The complex $(\Omega^\bullet_e,d)$ is called the de Rham complex. The homology of the de Rham complex is described by the Cartier isomorphism $$ H^*((\Omega^*_*(V),d))\simeq \Omega^*_*(V^{(1)}),$$ where $V^{(1)}$ is the image of a vector space $V$ under the Frobenius twist functor (c.f. \cite{cartier}). The de Rham and Koszul differentials are related by the formula $d\kappa+\kappa d=e\cdot id$. In particular, if $e$ is divisible by $p$ then 
\begin{equation}
\label{contractible_omegas}
d\kappa+\kappa d=0.
\end{equation}
Let $K^i_e=\ker \kappa_i.$ The equality (\ref{contractible_omegas}) shows that the complex $(K^\bullet_{pe},d)$ is a subcomplex of $(\Omega_{pe}^\bullet,d).$ This complex will be called the Koszul kernel complex. More important for us than the Cartier isomorphism recalled above will be the restriction of this isomorphism to the Koszul kernel complex:
\begin{equation}
\label{cartier_kernels}
H^*(K^\bullet_{pe}(V))\simeq K^*_e(V^{(1)}).
\end{equation}
In the sequel we omit subscripts, since we consider only the case $e=p.$

It is known that for $e<p$ the category $\mathcal{P}_e$ is a semisimple category, hence the first non-trivial case from the point of view of homological algebra is $e=p,$ being our object of interest in this paper. In this case it is sufficient to restrict our attention to hooks, i.e. Young diagrams $(m,1^n).$ Indeed, the block structure of $\mathcal{P}_e$ is similar to that of symmetric group: the functors $S_\lambda$ and $S_\mu$ are in the same block if and only if $\lambda$ and $\mu$ have the same $p$-cores (c.f. \cite{donkin}). The statement is also true for Weyl and simple functors. For $e=p$ straightforward calculations lead to the conclusion that different diagrams, which are not hooks, have different $p$-cores and all hooks have the same $p$-core, namely the empty diagram. $\operatorname{Ext}$-groups between functors being in different blocks are the trivial groups, hence the only non-trivial homological computations appear in the block of $\mathcal{P}_p$ corresponding to hooks denoted by $\mathcal{P}^\varnothing_p.$ Let $S_i=S_{(i+1,1^{p-i-1})}$, $W_i=W_{(i+1,1^{p-i-1})}$, $F_i=F_{(i+1,1^{p-i-1})}$ for $0\leq i \leq p-1.$

In \cite{weyman} the Schur functors are defined in an equivalent way from which one can easily conclude that $S_i$ is the cokernel of the Koszul differential $\kappa_{i+2}:\Omega^{i+2} \to \Omega^{i+1}$ for all $0\leq i \leq p-2$ (c.f (2.1.3) Example (h), \cite{weyman}). For $i=p-1$ we have $S_{p-1}=\Lambda^p=\Omega^p$. By the acyclicity of the Koszul complex $(\Omega^\bullet,\kappa)$ we see that $S_i$ is indeed the kernel of the Koszul differential $\kappa_i:\Omega^{i}\to \Omega^{i-1},$ i.e. $S_i=K^i$, and we obtain the short exact sequences
\begin{equation}
\label{koszul_sh_seq}
0 \to S_i \to \Omega^i \to S_{i-1} \to 0
\end{equation}
for $0\leq i \leq p-1$, assuming $S_{-1}=S_p=0$. It is worth noting that these short exact sequences can also be derived from the Littlewood-Richardson rule. Now we describe injective hulls and projective covers of simples in the block $\mathcal{P}_p^\varnothing.$
\begin{prop}
\label{injective_hull}
Let $0\leq i \leq p-1$ and let $I_i$ and $P_i$ be, respectively, the injective hull and the projective cover of the simple functor $F_i$. Then $I_i=\Omega^i$ and $P_i=(\Omega^i)^{\#}$. In particular, $I_i\simeq P_i$ unless $i=0$.
\end{prop}
\begin{proof}
Fix $0\leq i \leq p-1.$ Since $F_i$ is a subfunctor of $S_i$ and $S_i=\ker(\kappa_i:\Omega^i\to \Omega^{i-1})$, $F_i$ is a subfunctor of $\Omega^i.$ We observe that $\Lambda^i$ is a direct summand of $I^{\otimes i}$ and $S^{(p-i,1^i)}\simeq S^{p-i}\otimes I^{\otimes i}$ is an injective functor as a direct summand of the injective functor $S^p(\operatorname{Hom}(k^p,-))$ (c.f. Theorem 2.10,\cite{fs}). Hence $\Omega^i=S^{p-i}\otimes \Lambda^i$ is also injective. To end the proof, it is sufficient to show that $\Omega^i$ is an indecomposable functor. Then $I_i=\Omega^i$ and, by the Kuhn duality, $P_i=(\Omega^i)^{\#}$. The last statement follows from the first part and the isomorphism $S^i\simeq \Gamma^i$ unless $i=0$.

The functor $\Omega^0=S^p=S_0$ is indecomposable by the axioms of a highest weight category. We now assume that $1\leq i \leq p-1$ and we suppose that $\Omega^i$ is a decomposable functor, i.e. there exist nonzero subfunctors $G_1,G_2$ of $\Omega^i$ such that $\Omega^i=G_1 \oplus G_2.$ Since a direct summand of an injective object is injective, the functors $G_1$ and $G_2$ are injective. Then it follows that $\operatorname{Ext}^1(W_\mu,G_k)=0$ for $k=1,2$ and a Young diagram $\mu$ of weight $p.$ Therefore $G_1$ and $G_2$ are filtered by Schur functors. By the long exact sequence for $\operatorname{Ext}^*(W_\mu,-)$ applied to the short exact sequence $0\to G_1\to \Omega^i \to G_2 \to 0$ and the injectivity of $G_1, G_2$ and $\Omega^i$ we obtain the short exact sequence $0 \to \operatorname{Hom}(W_\mu,G_1) \to \operatorname{Hom}(W_\mu,\Omega^i) \to \operatorname{Hom}(W_\mu,G_2) \to 0$. Let us recall that for a functor $H$ filtered by Schur functors the multiplicity $(H:S_\mu)$ of the Schur functor $S_\mu$ in a corresponding filtration of $H$ is $(H:S_\mu) = \dim \operatorname{Hom}(W_\mu,H)$. It follows from this formula, the last short exact sequence and the fact that $\Omega^i$ has the filtration by Schur functors given by (\ref{koszul_sh_seq}) that $$(G_1:S_\mu)+(G_2:S_\mu)=(\Omega^i:S_\mu)=\begin{cases} 1 \quad \text{if } \mu=(i+1,1^{p-i-1}) \lor \mu=(i,1^{p-i})\\ 0 \quad otherwise   \end{cases}.$$ Then we conclude that
\begin{equation}
\label{multiplicity_summands}
G_1=S_i, G_2=S_{i-1} \quad \text{or} \quad G_1=S_{i-1}, G_2=S_i.
\end{equation}
We claim that the only injective Schur functor in $\mathcal{P}_p^\varnothing$ is $S^p.$ Indeed, we have seen above that $S^p=\Omega^0$ is injective. By the long exact sequence for $\operatorname{Ext}^*(I^{(1)},-)$ applied to the short exact sequence (\ref{koszul_sh_seq}) and the Vanishing Theorem (c.f. Theorem 2.13, \cite{fs}, this is a special case of the sum-diagonal adjunction mentioned in the Introduction) we see that $\operatorname{Ext}^i(I^{(1)},S_i)\simeq \operatorname{Ext}^{i-1}(I^{(1)},S_{i-1})$. Since $\operatorname{Hom}(I^{(1)},S^p)=k$ (c.f. Theorem 2.10, \cite{fs}), we conclude, by induction on $i$, that $S_i$ is not injective for $i>0.$  Thus, $S^p$ is the only injective functor in $\mathcal{P}^\varnothing_p$. By (\ref{multiplicity_summands}) it leads to the contradiction with the injectivity of both the functors $G_1$ and $G_2$. In particular, $\Omega^i$ is indecomposable.
\end{proof}
\begin{cor}[The decomposition matrix]~\\
\label{decomp_matrix}
The decomposition matrix $D=(d_{lm})$ of the category $\mathcal{P}^{\varnothing}_p$ is given by 
\begin{displaymath}
d_{lm}:=[W_{l-1}:F_{m-1}]=\begin{cases} 1 & if \ m=l \ or \ m=l+1 \\ 0 & otherwise \end{cases}.
\end{displaymath} 
Moreover, for all $0\leq j \leq p-1$ the following short exact sequences hold, assuming $F_p=0$:
\begin{align}
\label{richardson_seq}
 0 \to F_{j+1} \to W_j \to F_{j} \to 0 \\
\label{richardson_seq_dual}
 0 \to F_{j} \to S_j \to F_{j+1} \to 0.
\end{align}
\end{cor}
\begin{proof}
It follows from Proposition \ref{injective_hull} and the short exact sequence (\ref{koszul_sh_seq}) that $I_0\simeq S_0$ and $I_j$ is filtered by the Schur functors $S_j$ and $S_{j-1}$ for $1\leq j \leq p-1$. By the Brauer-Humphreys reciprocity (c.f. Theorem 3.11, \cite{cps}) $W_j$ has the composition series containing only two quotients, which are isomorphic to $F_j$ and $F_{j+1},$ i.e. $$ [W_j:F_i]=\begin{cases} 1 & \text{if } i=j \text{ or } i=j+1 \\ 0 & \text{otherwise} \end{cases} $$ for $0\leq i,j \leq p-1.$ It gives us immediately the decomposition matrix of $\mathcal{P}^{\varnothing}_p$ and it easily implies the following short exact sequence $$0\to F_{j+1} \to W_j \to F_j \to 0$$for $0\leq j \leq p-1$. By the Kuhn duality we obtain the second short exact sequence given in the assertion of the corollary.
\end{proof}
By the short exact sequences (\ref{richardson_seq}) and (\ref{richardson_seq_dual}) we obtain, respectively, the long exact sequence
\begin{align}
\label{koszul_res_dual}
& 0\leftarrow F_j \leftarrow W_j \leftarrow W_{j+1} \leftarrow \ldots \leftarrow W_{p-1} \leftarrow 0,\\
\label{koszul_res}
& 0\to F_j \to S_j \to S_{j+1} \to \ldots \to S_{p-1} \to 0
\end{align}
for all $0\leq j \leq p-1.$

Fix $0\leq i \leq p-1$. The long exact sequence 
\begin{equation}
\label{inj_res_schur}
0\to S_i\to \Omega^i \xrightarrow{\kappa_i} \Omega^{i-1} \xrightarrow{\kappa_{i-1}} \ldots \xrightarrow{\kappa_2} \Omega^1 \xrightarrow{\kappa_1} \Omega^0 \to 0
\end{equation}
is a minimal injective resolution of $S_i$. Indeed, we have $\operatorname{coker} \kappa_j=S_{j-2}$ for $1\leq j\leq p-1$ (assuming $S_{-1}=0$) by the acyclicity of the Koszul complex. The injective hull of $S_j$ is $\Omega^j$ by the fact that $F_j$ is the socle of $S_j$ and Proposition \ref{injective_hull}. Thus, (\ref{inj_res_schur}) is indeed the minimal injective resolution of $S_i$.

Now let us consider the following long exact sequence, which is constructed by gluing the Kuhn dual complex of the Koszul complex and the truncated Koszul complex:
\begin{multline}
\label{minimal_projres_schur}
0\to (\Omega^0)^{\#}\xrightarrow{\kappa_1^{\#}} (\Omega^1)^{\#}\to \ldots \to (\Omega^{p-2})^{\#}\xrightarrow{\kappa_{p-1}^{\#}} (\Omega^{p-1})^{\#}\xrightarrow{\kappa_p\kappa_p^{\#}} \Omega^{p-1}\xrightarrow{\kappa_{p-1}} \Omega^{p-2}\to \ldots \to\Omega^{i+2}\xrightarrow{\kappa_{i+2}} \\ \to \Omega^{i+1}\to S_i \to 0.
\end{multline}
We claim that (\ref{minimal_projres_schur}) is the minimal projective resolution of $S_i$. At this point, I would like to thank the referee for pointing out this resolution. Let us turn to the proof of our claim. Indeed, the projective cover of $S_{p-1}=F_{p-1}$ is $(\Omega^{p-1})^{\#}\simeq\Omega^{p-1}$ by Proposition \ref{injective_hull}. For $0\leq j<p-1$ we have an epimorphism $\Omega^{j+1}\to S_j$ by (\ref{koszul_sh_seq}) and we see from the proof of Proposition \ref{injective_hull} that $\Omega^{j+1}\simeq (\Omega^{j+1})^{\#}$ is projective and indecomposable, hence $\Omega^{j+1}$ is the projective cover of $S_j$. We observe that $\ker \kappa^{\#}_j=(\operatorname{coker}\kappa_j)^{\#}=(S^{j-2})^{\#}=W_{j-2}$ for $1\leq j\leq p$ (assuming $W_{-1}=0$), by the Kuhn duality and the acyclicity of the Koszul complex, and $\ker \kappa_p\kappa_p^{\#}=\ker \kappa_p^{\#}$, since $\kappa_p$ is a monomorphism. The projective cover of $W_j$ is $(\Omega^j)^{\#}$ by the fact that $F_j$ is the top of $W_j$ and Proposition \ref{injective_hull}. Thus, (\ref{minimal_projres_schur}) is indeed the minimal projective resolution of $S_i$.
\section{Additive $\operatorname{Ext}$-computations}
Since all our computations of $\operatorname{Ext}$-groups are performed in the category $\mathcal{P}_p^\varnothing$, we ease our notation by omitting subscripts in our $\operatorname{Ext}$-groups. We will use the following formula for the $\operatorname{Ext}$-groups between $W_\lambda$ and $S_\mu$ (c.f. Lemma 2.1, \cite{krause}):
\begin{equation}
\label{ext_weyl_schur}
\operatorname{Ext}^q(W_\lambda,S_\mu)=\begin{cases} k & \text{if} \ q=0 \land \lambda=\mu \\ 0 & \text{otherwise}  \end{cases}.
\end{equation}

\begin{prop} 
\label{ext_simple_schur}
If $0\leq m,n \leq p-1$ then
\begin{displaymath}
\operatorname{Ext}^q(F_m,S_n)=\begin{cases} k \quad \text{if} \ n\geq m \land q=n-m \\ 0 \quad \text{otherwise} \end{cases}.
\end{displaymath}
\end{prop}
\begin{proof}
We have by (\ref{ext_weyl_schur}) that $R^qG(W_j)=0$ for $q>0$, $0\leq j \leq p-1$ and $G=\operatorname{Hom}(-,S_n).$ In other words, the functors $W_j$ are $G$-acyclic. By that and the long exact sequence (\ref{koszul_res_dual}) we deduce that $\operatorname{Ext}^q(F_m,S_n)=R^q\operatorname{Hom}(W_{\bullet},S_n)$, where $\operatorname{Hom}(W_\bullet, S_n)$ is the complex obtained by applying the functor $\operatorname{Hom}(-,S_n)$ to the resolution of $F_m$ (\ref{koszul_res_dual}) and erasing the first term. Then the assertion of the proposition follows from (\ref{ext_weyl_schur}).
\end{proof}
\begin{remark}
One can also use in the proof of Proposition \ref{ext_simple_schur} the injective resolution of $S_n$ (\ref{inj_res_schur}) instead of the resolution of $F_m$ by Weyl functors (\ref{koszul_res_dual}).
\end{remark}
\begin{cor}
\label{ext_weyl_simple}
If $0\leq m,n \leq p-1$ then
\begin{displaymath}
\operatorname{Ext}^q(W_m,F_n)=\begin{cases} k \quad \text{if} \ m\geq n \land q=m-n \\ 0 \quad otherwise \end{cases}.
\end{displaymath}
\end{cor}
\begin{proof}
The assertion of the corollary follows from Proposition \ref{ext_simple_schur} by the Kuhn duality.
\end{proof}
We recall that a highest weight category $\mathcal{C}$ with poset $P$ has a Kazhdan-Lusztig theory relative to a length function $l:P\to \mathbb{Z}$ if and only if the following conditions hold:
\begin{equation*}
\begin{split}
\operatorname{Ext}^i(F_\lambda,S_\mu)\neq 0 \implies i \equiv l(\lambda)-l(\mu) \ \text{mod } 2, \\
\operatorname{Ext}^i(W_\lambda,F_\mu)\neq 0 \implies i \equiv l(\lambda)-l(\mu) \ \text{mod } 2
\end{split}
\end{equation*}
for any $\lambda,\mu \in P$ (c.f. Theorem 2.4, \cite{cps_kazlus}).
\begin{cor}
\label{kltheory}
The category $\mathcal{P}_p$ has a Kazhdan-Lusztig theory relative to the function $l$ given by $l((i+1,1^{p-i-1}))=i$ and $l(\lambda)=0$ for a Young diagram $\lambda$ not being a hook.
\end{cor}
\begin{proof}
The assertion of the corollary follows immediately by using the block structure of $\mathcal{P}_p$ described in Section 1, Proposition \ref{ext_simple_schur} and Corollary \ref{ext_weyl_simple}.
\end{proof}
\begin{cor}
\label{ext_simples}
If $0\leq m,n \leq p-1$ then
\begin{displaymath}
\operatorname{Ext}^q(F_m,F_n)=\begin{cases} k \quad \text{if} \ q=|m-n|+2r, \ \text{where} \ 0\leq r \leq p-\max\{m,n\}-1 \\ 0 \quad \text{otherwise} \end{cases}.
\end{displaymath}
\end{cor}
\begin{proof}
If a highest weight category $\mathcal{C}$ with poset $P$ has a Kazhdan-Lusztig theory then the following equality holds: 
\begin{equation}
\label{kl_eq}
\dim \operatorname{Ext}^q(F_\lambda,F_\mu)=\sum_{\tau\in P} \sum_{i+j=q} \dim \operatorname{Ext}^i(F_\lambda,S_\tau)\cdot \dim\operatorname{Ext}^j(W_\tau,F_\mu)
\end{equation}
for $\lambda,\mu\in P$ (c.f. Theorem 3.5, \cite{cps_kazlus}). The category $\mathcal{P}_p$ has a Kazhdan-Lusztig theory by Corollary \ref{kltheory}. In view of the block structure of $\mathcal{P}_p$ we restrict the poset of that category to the poset of hooks. By the Kuhn duality $\operatorname{Ext}^q(W_i,F_j)=\operatorname{Ext}^q(F_j,S_i)$ for $0\leq i,j \leq p-1$. Thus, in the case of our category we can rewrite the equality (\ref{kl_eq}) in the following form:
\begin{displaymath}
\dim \operatorname{Ext}^q(F_m,F_n)=\sum_{0\leq i \leq p-1} \sum_{q_1+q_2=q} \dim\operatorname{Ext}^{q_1}(F_m,S_i)\cdot \dim\operatorname{Ext}^{q_2}(F_n,S_i).
\end{displaymath}
If $i<m$ or $i<n$ then $\dim \operatorname{Ext}^{q_1}(F_m,S_i)\cdot \dim \operatorname{Ext}^{q_2}(F_n,S_i) = 0$ by Proposition \ref{ext_simple_schur}. If $i\geq m$ and $i\geq n,$ i.e. $i\geq \max\{m,n\},$ then $\dim\operatorname{Ext}^{q_1}(F_m,S_i)\cdot\dim \operatorname{Ext}^{q_2}(F_n,S_i) \neq 0$ if and only if $q_1=i-m \text{ and } q_2=i-n$, by Proposition \ref{ext_simple_schur}. By the above observations we obtain
\begin{equation}
\label{kl_eqpol2}
\begin{split}
& \dim\operatorname{Ext}^q(F_m,F_n)=\sum_{0\leq i \leq p-1} \sum_{q_1+q_2=q} \dim\operatorname{Ext}^{q_1}(F_m,S_i)\cdot\dim\operatorname{Ext}^{q_2}(F_n,S_i)=\\& =\begin{cases} 1 & \text{if} \ q_1=i-m, \ q_2=i-n, \ \text{i.e.} \ q=2i-m-n \text{ for } \max\{m,n\}\leq i \leq p-1 \\ 0 & \text{otherwise} \end{cases}.
\end{split}
\end{equation} 
We note that $2i-m-n=\max\{m,n\}-\min\{m,n\}+2(i-\max\{m,n\})=|m-n|+2(i-\max\{m,n\}),$ hence $q=|m-n|+2r$ for $0\leq r \leq p-\max\{m,n\}-1.$ The assertion of the corollary follows from (\ref{kl_eqpol2}) and the last remark.
\end{proof}
\begin{lem}
\label{hom_schur_omega}
If $0\leq m,n \leq p-1$ then
\begin{displaymath}
\operatorname{Hom}(S_m,\Omega^n)=\operatorname{Hom}((\Omega^n)^{\#},W_m)=\begin{cases} k \quad \text{if } n=m \lor n=m+1 \\ 0 \quad \text{otherwise} \end{cases}.
\end{displaymath}
\end{lem}
\begin{proof}
We have $S_{p-1}=F_{p-1}$, hence $\dim\operatorname{Hom}(S_{p-1},\Omega^n)=\dim\operatorname{Hom}(F_{p-1},\Omega^n)=\delta_{p-1,n}$. In case $m<p-1$ we apply the exact functor $\operatorname{Hom}(-,\Omega^n)$ to the short exact sequence (\ref{richardson_seq_dual}) for $j=m$. Then we indeed have $$\dim \operatorname{Hom}(S_m,\Omega^n)=\dim\operatorname{Hom}(F_{m+1},\Omega^n)+\dim\operatorname{Hom}(F_m,\Omega^n)=\begin{cases} 1 \quad \text{if } n=m \lor n=m+1 \\ 0 \quad \text{otherwise}  \end{cases}.$$The equality $\operatorname{Hom}(S_m,\Omega^n)=\operatorname{Hom}((\Omega^n)^{\#},W_m)$ is immediate by the Kuhn duality.
\end{proof}
\begin{prop}
\label{ext_schur}
If $0\leq m,n \leq p-1$ then
\begin{displaymath}
\operatorname{Ext}^q(S_m,S_n)=\begin{cases} k \quad if \ (n=m \land q=0) \lor (n>m \land (q=n-m-1 \lor q=n-m)) \\ 0 \quad otherwise \end{cases}.
\end{displaymath}
\end{prop}
\begin{proof}
We use the injective resolution of $S_n$ (\ref{inj_res_schur}), hence $\operatorname{Ext}^q(S_m,S_n)=H^q(\operatorname{Hom}(S_m,\Omega^\bullet))$, where $\operatorname{Hom}(S_m,\Omega^\bullet)$ is the complex obtained by applying the functor $\operatorname{Hom}(S_m,-)$ to the resolution of $S_n$ (\ref{inj_res_schur}) with the removed first term.  We will now show that in case $m<n$ the differential $\operatorname{Hom}(S_m,\Omega^{m+1})=k\xrightarrow{\kappa_{m+1}\circ -} \operatorname{Hom}(S_m,\Omega^m)=k$ is the zero map. The map $d_m\imath_m$, where $\imath_m$ is the inclusion $S_m \hookrightarrow \Omega^m$, generates $\operatorname{Hom}(S_m,\Omega^{m+1})$. The considered differential sends an element $c\cdot d_m\imath_m $ with $c\in k$ to the element $c\cdot \kappa_{m+1}d_m\imath_m=-c\cdot d_{m-1}\kappa_{m}\imath_m=0$, i.e. this differential is indeed the zero map. We used here (\ref{contractible_omegas}) in the first equality. Then the assertion of the proposition follows from Lemma \ref{hom_schur_omega} and the above observation.
\end{proof}
\begin{cor}
\label{ext_weyls}
If $0\leq m,n \leq p-1$ then
\begin{displaymath}
\operatorname{Ext}^q(W_m,W_n)=\begin{cases} k \quad if \ (n=m \land q=0) \lor (n<m \land (q=m-n-1 \lor q=m-n)) \\ 0 \quad otherwise \end{cases}.
\end{displaymath}
\end{cor}
\begin{proof}
The assertion of the corollary follows from Proposition \ref{ext_schur} by the Kuhn duality.
\end{proof}
As a corollary to the last proposition, we claim that the long exact sequence (\ref{koszul_res}) is that induced by the Koszul kernel complex, possibly truncated. Indeed, we saw in the previous section that $S_j=K^j$ for $0\leq j \leq p-1.$ Letting $q=0$ in Proposition \ref{ext_schur}, we deduce that the long exact sequence (\ref{koszul_res}) for $0\leq i \leq p-1$ has the following form: $$ 0 \to F_i \to S_i \xrightarrow{d_i} S_{i+1} \to \ldots \to S_{p-2} \xrightarrow{d_{p-2}} S_{p-1} \to 0,$$ i.e. this is indeed the long exact sequence induced by the Koszul kernel complex. We also obtain \begin{equation}
\label{kernel_koszul_complex}
F_i=\ker(S_i\xrightarrow{d_i} S_{i+1})
\end{equation}
for $0\leq i \leq p-1.$
\begin{prop}
\label{ext_schur_simple}
If $0\leq m,n \leq p-1$ then $$\operatorname{Ext}^q(S_m,F_n)=\begin{cases} k & \text{if} \ (m<n \land q=n-m-1) \lor q=2p-m-n-2 \\ 0 & \text{otherwise} \end{cases}.$$
\end{prop}
\begin{proof}
We use the minimal projective resolution of $S_m$ (\ref{minimal_projres_schur}), hence $\operatorname{Ext}^q(S_m,F_n)=H^q(\operatorname{Hom}(P_\bullet,F_n))$, where $\operatorname{Hom}(P_\bullet,F_n)$ is the complex obtained by applying the functor $\operatorname{Hom}(-,F_n)$ to the resolution of $S_m$ (\ref{minimal_projres_schur}) with the removed first term. Then the assertion of the proposition follows from Lemma \ref{hom_schur_omega} and the fact that the differentials of the complex $\operatorname{Hom}(P_\bullet,F_n)$ are the zero maps since the resolution (\ref{minimal_projres_schur}) is minimal.
\end{proof}
\begin{cor}
If $0\leq m,n \leq p-1$ then $$\operatorname{Ext}^q(F_m,W_n)=\begin{cases} k & \text{if} \ (m>n \land q=m-n-1) \lor q=2p-m-n-2 \\ 0 & \text{otherwise} \end{cases}.$$
\end{cor}
\begin{proof}
The assertion of the corollary follows from Proposition \ref{ext_schur_simple} by the Kuhn duality.
\end{proof}
\begin{prop}
If $0\leq m,n \leq p-1$ then
$$\operatorname{Ext}^q(S_m,W_n)=\begin{cases} k & \text{if} \ (m=n \land q=0)\lor q=2p-m-n-3 \lor q=2p-m-n-2  \\ 0 & \text{otherwise} \end{cases}.$$
\end{prop}
\begin{proof}
We use the minimal projective resolution of $S_m$ (\ref{minimal_projres_schur}), hence $\operatorname{Ext}^q(S_m,W_n)=H^q(\operatorname{Hom}(P_\bullet,W_n))$, where $\operatorname{Hom}(P_\bullet,W_n)$ is the complex obtained by applying the functor $\operatorname{Hom}(-,W_n)$ to the resolution of $S_m$ (\ref{minimal_projres_schur}) with the removed first term.
It was showed in the proof of Proposition \ref{ext_schur} that the differential $\operatorname{Hom}(S_n,\Omega^{n+1})\to \operatorname{Hom}(S_n,\Omega^n)$ in the complex $\operatorname{Hom}(S_m,\Omega^\bullet)$ considered there is the zero map, hence the differential $\operatorname{Hom}((\Omega^{n+1})^{\#},W_n)\xrightarrow{-\circ \kappa^{\#}_{n+1}} \operatorname{Hom}((\Omega^n)^{\#},W_n)$ is the zero map by the Kuhn duality. Now we claim that in case $m<n<p-1$ the differential $\operatorname{Hom}(\Omega^n,W_n)=k\xrightarrow{-\circ \kappa_{n+1}} \operatorname{Hom}(\Omega^{n+1},W_n)=k$ is an isomorphism. Indeed, the projection $s_n:\Omega^n \twoheadrightarrow W_n$ generates $\operatorname{Hom}(\Omega^n,W_n)$. The above differential sends an element $c\cdot s_n$ with $c\in k$ to the element $c\cdot s_n \kappa_{n+1}$. Let us suppose that $s_n\kappa_{n+1}=0$, i.e. $\operatorname{im}\kappa_{n+1} \subset \ker s_n$. We see that $\operatorname{im}\kappa_{n+1}=S_n$ by the acyclicity of the Koszul complex and $\ker s_n=W_{n-1}$ by the Kuhn dual of (\ref{koszul_sh_seq}), hence $S_n\subset W_{n-1}$. On the other hand, the composition series of $S_n$ is not contained in the composition series of $W_{n-1}$, contrary to the last inclusion. In consequence, $s_n\kappa_{n+1}\neq 0$, hence the differential under consideration is non-zero, and the claim follows. In case $n=p-1$ we have $W_{p-1}=F_{p-1}$, hence the differential $\operatorname{Hom}(\Omega^{p-1},W_{p-1})\to \operatorname{Hom}((\Omega^{p-1})^{\#},W_{p-1})$ is the zero map since this is a differential in the complex $\operatorname{Hom}(P_\bullet,F_{p-1})$ considered in Proposition \ref{ext_schur_simple}. We also observe that in case $n=p-2$ the differential $\operatorname{Hom}(\Omega^{p-1},W_{p-2})=k\xrightarrow{- \circ \kappa_p\kappa^{\#}_p} \operatorname{Hom}((\Omega^{p-1})^{\#},W_{p-2})=k$ is the zero map. Indeed, the map $s_{p-2}\kappa_{p-1}$ generates $\operatorname{Hom}(\Omega^{p-1},W_{p-2})$. The considered differential sends an element $c\cdot s_{p-2}\kappa_{p-1}$ with $c\in k$ to the element $c\cdot  s_{p-2}\kappa_{p-1}\kappa_p\kappa^{\#}_p=0$, i.e. this differential is the zero map, as required. Finally, the assertion of the proposition follows from Lemma \ref{hom_schur_omega} and all the above observations.
\end{proof}
\section{The Yoneda algebras of the Schur functors and of simple functors}
\subsection{The Yoneda algebra of the Schur functors}
The main aim of this subsection is to describe the structure of the Yoneda algebra of Schur functors and to prove that the endomorphism algebra whose cohomology algebra is the Yoneda algebra of Schur functors is a formal DG algebra.

Let $\mathcal{T}_i$ denote the complex obtained by removing $S_i$ in the injective resolution of $S_i$ (\ref{inj_res_schur}). We will now determine a basis of the Yoneda algebra of Schur functors whose elements are induced by the morphisms of complexes defined in the following way. For given $0\leq i\leq j \leq p-1$ we define the morphism of complexes $\gamma_{ji}\in \operatorname{Hom}^{j-i}(\mathcal{T}_i,\mathcal{T}_j)$ as the natural inclusion. For $0 \leq i \leq p-1$ let $\widetilde{d}_i:\mathcal{T}_i\to \mathcal{T}_{i+1}$ be the map whose components are given by $\widetilde{d}_i^m=(-1)^{i-m}d_{i-m}$ for $0\leq m \leq i.$ By (\ref{contractible_omegas}) we see that $\widetilde{d}_i$ is indeed a morphism of complexes. Then for given $0\leq i < j \leq p-1$ let $\overline{\gamma}_{ji} \in \operatorname{Hom}^{j-i-1}(\mathcal{T}_i,\mathcal{T}_j)$ be the morphism of complexes given by $\overline{\gamma}_{ji}=\gamma_{j,i+1} \widetilde{d}_i.$ We now show that $\gamma_{ji}$ and $\overline{\gamma}_{ji}$ are not null-homotopic.
\begin{lem}
\label{notnull_schur}
The morphism of complexes $\gamma_{ji}$ (resp. $\overline{\gamma}_{ji}$) is not null-homotopic for any $0\leq i\leq j\leq p-1 $ (resp. $0\leq i < j \leq p-1$).
\end{lem}
\begin{proof}
We prove first the assertion of the lemma for $\gamma_{ji}.$ Fix $0\leq i \leq j \leq p-1.$ We suppose that $\gamma_{ji}$ is null-homotopic. Let $h_{ji}=h_{ji}^*$ be a homotopy between $\gamma_{ji}$ and the zero map. We consider the following diagram, where the top row is the complex $\mathcal{T}_i$ and the bottom row is the complex $\mathcal{T}_j[j-i].$
$$\begin{tikzcd}[column sep=3em,row sep=3em]
 & \ldots \arrow[r] & \Omega^0 \arrow[r] \arrow[ld,"h^i_{ji}"] \arrow[d,"\gamma^{i}_{ji}"] &  0 \arrow[r]\arrow[ld,"h^{i+1}_{ji}"]  & \ldots \\
\ldots \arrow[r] & \Omega^1 \arrow[r,"\kappa_1"] & \Omega^0  \arrow[r] &  \ldots & 
\end{tikzcd}$$
Clearly $h^{i+1}_{ji}=0$ and, by definition, $\gamma^i_{ji}=id.$ Since $h_{ji}$ is a homotopy between $\gamma_{ji}$ and the zero map, we obtain $\kappa_1h^i_{ji}=id$. By Lemma \ref{hom_schur_omega} we have $\operatorname{Hom}(\Omega^0,\Omega^1)=\operatorname{Hom}(S_0,\Omega^1)=k.$ It implies that $h^i_{ji}=a\cdot d_0$ for some $a\in k,$ hence
\begin{equation}
\label{homotopy_schurs1}
\kappa_1\circ (a \cdot d_0)=id.
\end{equation}
On the other hand, $\kappa_1\circ (a \cdot d_0)=a\cdot \kappa_1 d_0=0$ by (\ref{contractible_omegas}), contrary to (\ref{homotopy_schurs1}). Therefore $\gamma_{ji}$ is not null-homotopic.

Now we turn to the case of the map  $\overline{\gamma}_{ji}.$ Fix $0\leq i < j \leq p-1.$ We suppose that $\overline{\gamma}_{ji}$ is null-homotopic. Let $\overline{h}_{ji}=\overline{h}^{*}_{ji}$ be a homotopy between $\overline{\gamma}_{ji}$ and the zero map. Then there is the following diagram, where the top row is the complex $\mathcal{T}_i$ and the bottom row is the complex $\mathcal{T}_j[j-i-1].$
$$\begin{tikzcd}[column sep=3em,row sep=3em]
 & \ldots \arrow[r] & \Omega^0 \arrow[r]\arrow[ld,"\overline{h}^i_{ji}"] \arrow[d,"\overline{\gamma}^{i}_{ji}"] &  0 \arrow[r]\arrow[ld,"\overline{h}^{i+1}_{ji}"] & \ldots \\
\ldots \arrow[r] & \Omega^2 \arrow[r,"\kappa_2"] & \Omega^1  \arrow[r] &  \ldots &
\end{tikzcd}$$
Clearly $\overline{h}^{i+1}_{ji}=0$ and, by definition, $\overline{\gamma}^i_{ji}=d_0.$ Since $\overline{h}_{ji}$ is a homotopy between $\overline{\gamma}_{ji}$ and the zero map, the following equality holds:
\begin{equation}
\label{homotopy_schurs2}
\kappa_2\overline{h}^i_{ji}=d_0.
\end{equation}
By Lemma \ref{hom_schur_omega} we have $\operatorname{Hom}(\Omega^0,\Omega^2)=\operatorname{Hom}(S_0,\Omega^2)=0.$ It implies that $\overline{h}^i_{ji}=0,$ hence $\kappa_2\overline{h}^i_{ji}=0,$ contradicting (\ref{homotopy_schurs2}). Thus, $\overline{\gamma}_{ji}$ is not null-homotopic.
\end{proof}
The class of the map $\gamma_{ji}$ (resp. $\overline{\gamma}_{ji}$) in $\operatorname{Ext}^{j-i}(S_i,S_j)$ (resp. $\operatorname{Ext}^{j-i-1}(S_i,S_j)$) will be denoted by $[\gamma_{ji}]$ (resp. $[\overline{\gamma}_{ji}]$). Set $S=\bigoplus_{0\leq i \leq p-1} S_i.$ By Proposition \ref{ext_schur} and Lemma \ref{notnull_schur} we obtain
\begin{equation}
\label{vectorspacestruc_extschur2}
\operatorname{Ext}^*(S,S)=\operatorname{span}(\{[\gamma_{ji}]: 0\leq i \leq j \leq p-1\}\cup \{[\overline{\gamma}_{ji}]:0\leq i < j \leq p-1\})
\end{equation}
as the graded vector space. By the definition of the maps $\gamma_{ji}$ we have
\begin{equation}
\label{productschur_1}
\gamma_{ml}\cdot \gamma_{ji}=\begin{cases} \gamma_{mi} & \text{if } j=l \\ 0 & \text{otherwise} \end{cases}.
\end{equation}
We also find that
\begin{equation}
\label{productschur_2}
\gamma_{ml}\cdot \overline{\gamma}_{ji}=\begin{cases} \overline{\gamma}_{mi} & \text{if } j=l \\ 0 & \text{otherwise}  \end{cases}.
\end{equation}
Indeed, if $j=l$ then $$\gamma_{mj}\cdot \overline{\gamma}_{ji}=\gamma_{mj}\gamma_{j,i+1}\widetilde{d}_i=\gamma_{m,i+1}\widetilde{d}_i=\overline{\gamma}_{mi}.$$ We note that $ \gamma_{j,i+1} \widetilde{d}_i=\widetilde{d}_{j-1} \gamma_{j-1,i},$ hence, similarly to the above, we obtain 
\begin{equation}
\label{productschur_3}
\overline{\gamma}_{ml}\cdot \gamma_{ji}=\begin{cases} \overline{\gamma}_{mi} & \text{if } j=l \\ 0 & \text{otherwise}  \end{cases}.
\end{equation}
We also observe that
\begin{equation}
\label{productschur_4}
\overline{\gamma}_{ml}\cdot \overline{\gamma}_{ji}=0.
\end{equation}
Indeed, in case $j=l$ we have $$\overline{\gamma}_{ml}\cdot \overline{\gamma_{ji}}=\gamma_{m,j+1}\widetilde{d}_j\gamma_{j,i+1}\widetilde{d}_i=\gamma_{m,j+1}\gamma_{j+1,i+2}\widetilde{d}_{i+1}\widetilde{d}_i=0.$$ The equalities (\ref{productschur_1}) -- (\ref{productschur_4}) imply the equalities in the algebra $\operatorname{Ext}^*(S,S)$:
\begin{align*}
\nonumber & [\gamma_{ml}]\cdot [\gamma_{ji}]=\begin{cases} [\gamma_{mi}] & \text{if } j=l \\ 0 & \text{otherwise} \end{cases}, \qquad [\gamma_{ml}]\cdot [\overline{\gamma}_{ji}]=\begin{cases} [\overline{\gamma}_{mi}] & \text{if } j=l \\ 0 & \text{otherwise} \end{cases},\\ &
 [\overline{\gamma}_{ml}]\cdot [\gamma_{ji}]=\begin{cases} [\overline{\gamma}_{mi}] & \text{if } j=l \\ 0 & \text{otherwise} \end{cases}, \qquad [\overline{\gamma}_{ml}]\cdot [\overline{\gamma}_{ji}]= 0  
\end{align*}
with the appropriate inequality between the indices for a given equality. These equalities describe the multiplication on the Yoneda algebra of Schur functors $\operatorname{Ext}^*(S,S)$.

Let $\mathcal{T}=\bigoplus_{0\leq i \leq p-1} \mathcal{T}_i.$ We now show the DG formality of the DG algebra $\operatorname{End}^*(\mathcal{T})$ with the components $\operatorname{End}^i(\mathcal{T})=\prod_{n-m=i}\operatorname{Hom}_{\mathcal{P}_p^\varnothing}(\mathcal{T}^m,\mathcal{T}^n)$ and the differential $\mathfrak{d}$ given by $\mathfrak{d}(f^j)=(\kappa f^j-(-1)^i f^{j+1} \kappa)$ for $(f^j)\in \operatorname{End}^i(\mathcal{T}).$  We also describe the structure of the Yoneda algebra of Schur functors.
\begin{theorem}[]~
\label{quasiiso_ext_schurs} 
\begin{enumerate}[label=(\roman*)]
\item The algebra $\operatorname{End}^*(\mathcal{T})$ is a formal DG algebra, i.e. there exists a quasi-isomorphism of DG algebras $\phi:\operatorname{Ext}^*(S,S)\to \operatorname{End}^*(\mathcal{T}),$ where $\operatorname{Ext}^*(S,S)$ is the DG algebra with zero differential.
\item The Yoneda algebra of Schur functors $\operatorname{Ext}^*(S,S)$ is isomorphic to the square-zero extension of the algebra $U_p(k)$ of upper triangular matrices over $k$ by the $U_p(k)$-bimodule $U^{+}_p(k)$ of strictly upper triangular matrices over $k.$ 
\end{enumerate}
\end{theorem}
\begin{proof}
Let $A=\bigoplus_{t\in \mathbb{N}} A_t$ be the graded vector space with the components $$A_t=\operatorname{span}(\{a_{ji}:0\leq i \leq j \leq p-1 \ \land \ j-i=t\}\cup\{\overline{a}_{ji}:0\leq i < j \leq p-1 \land j-i-1=t\})$$ for $0\leq t \leq p-1$ and $A_t=0$ for $t\geq p,$ where $a_{ji},\overline{a}_{ji}$ are the formal symbols for $i,j$ satisfying the conditions given above. We define the multiplication on $A$ as follows:
\begin{align*}
& a_{ml}\cdot a_{ji}=\begin{cases} a_{mi} & \text{if } j=l \\ 0 & \text{otherwise} \end{cases}, \qquad a_{ml}\cdot \overline{a}_{ji}=\begin{cases} \overline{a}_{mi} & \text{if } j=l \\ 0 & \text{otherwise} \end{cases},\\ &
 \overline{a}_{ml}\cdot a_{ji}=\begin{cases} \overline{a}_{mi} & \text{if } j=l \\ 0 & \text{otherwise} \end{cases}, \qquad \overline{a}_{ml}\cdot \overline{a}_{ji}=0
\end{align*}
with the appropriate inequality between the indices for a given equality. The differential on $A$ is defined to be the zero differential.

Let us first prove (i). Let $\phi:A\to \operatorname{End}^*(\mathcal{T})$ be the map such that $\phi(a_{ji})=\gamma_{ji}$ for $0\leq i \leq j \leq p-1$ and $\phi(\overline{a}_{ji})=\overline{\gamma}_{ji}$ for $0\leq i < j \leq p-1.$ By (\ref{productschur_1})-(\ref{productschur_4}) we see that $\phi$ is a graded homomorphism and $\phi$ is a chain map, since $\phi(a_{ji})$ and $\phi(\overline{a}_{ji})$ are morphisms of complexes. In other words, $\phi$ is a DG homomorphism. It suffices to show that $\phi$ is a quasi-isomorphism. It follows from the definition of the differential on $A$ that $H^*(A)\simeq A.$ Then it is evident that the induced graded algebra homomorphism $\phi_*:H^*(A)\simeq A \to H^*(\operatorname{End}^*(\mathcal{T}))\simeq \operatorname{Ext}^*(S,S)$ is given by $\phi_*(a_{ji})=[\gamma_{ji}]$ and $\phi_*(\overline{a}_{ji})=[\overline{\gamma}_{ji}].$ By (\ref{vectorspacestruc_extschur2}) $\phi_*$ maps a basis of $A$ onto a basis of $\operatorname{Ext}^*(S,S).$ Hence $\phi_*$ is a graded algebra isomorphism. In particular, $\phi:A\simeq \operatorname{Ext}^*(S,S) \to \operatorname{End}^*(\mathcal{T})$ is a quasi-isomorphism, as required.

It remains to prove (ii). Let $\{e_{ij}:1\leq i \leq j \leq p\}$ (resp. $\{\overline{e}_{ij}:1\leq i < j \leq p\}$) be the standard basis of $U_p(k)$ (resp. $U^{+}_p(k)$). We will denote by $U_p(k)\oplus U_p^{+}(k)$ the square-zero extension of the algebra $U_p(k)$ by the $U_p(k)$-bimodule $U_p^{+}(k).$ We recall that the multiplication on $U_p(k)\oplus U_p^{+}(k)$ is given by $(u_1,\overline{u}_1)\cdot (u_2,\overline{u}_2)=(u_1u_2,u_1\overline{u}_2+\overline{u}_1u_2).$ Then it is immediate that
\begin{align}
\label{equalities_triangmatrix}
\nonumber & (e_{ij},0)\cdot (e_{lm},0)=\begin{cases} (e_{im},0) & \text{if } j=l, \\ (0,0) & \text{otherwise}\end{cases}, \qquad (e_{ij},0)\cdot(0,\overline{e}_{lm})=\begin{cases} (0,\overline{e}_{im}) & \text{if } j=l, \\ (0,0) & otherwise \end{cases}, \\ & (0,\overline{e}_{ij})\cdot (e_{lm},0)=\begin{cases} (0,\overline{e}_{ik}) & \text{if } j=l\\ (0,0) & \text{otherwise} \end{cases}, \qquad (0,\overline{e}_{ij})\cdot (0,\overline{e}_{lm})=(0,0)
\end{align}
with the appropriate inequality between the indices for a given equality. Let $\psi:A\to U_p(k)\oplus U_p^{+}(k)$ be the map given by $\psi(a_{ji})=(e_{i+1,j+1},0)$ for $0\leq i \leq j \leq p-1$ and $\psi(\overline{a}_{ji})=(0,\overline{e}_{i+1,j+1})$ for $0\leq i < j \leq p-1.$  By (\ref{equalities_triangmatrix}) we conclude that $\psi$ is an algebra homomorphism. Since $\psi$ maps a basis of $A$ onto a basis of $U_p(k)\oplus U_p^{+}(k)$, $\psi$ is an algebra isomorphism. Thus, we obtain $\operatorname{Ext}^*(S,S)\simeq A \simeq U_p(k)\oplus U_p^{+}(k),$ and the proof is complete.
\end{proof}
\begin{cor}
\label{derivedequiv_schurs}
There is an equivalence of triangulated categories $\mathcal{D}^b\mathcal{P}_p^{\varnothing} \simeq \mathcal{D}^b(\operatorname{Ext}^*(S,S)-\text{mod}^{\text{gr}}).$
\end{cor}
\begin{proof}
Let $\mathcal{T}^{\#}$ be the resolution obtained from $\mathcal{T}$ by using the Kuhn duality. Clearly $\mathcal{T}^{\#}$ is a cofibrant object in the category of complexes of $\mathcal{P}^{\varnothing}_p$. It is known that $\mathcal{T}^{\#}$ is a small generator of $\mathcal{D}^b\mathcal{P}^{\varnothing}_p$ (c.f. Lemma 5.5, \cite{krause2}). We consider the category $\mathcal{P}_p^{\varnothing}$ to be the DG category concentrated in degree $0.$ Then the equivalence $$\mathcal{D}^b\mathcal{P}^{\varnothing}_p\simeq \mathcal{D}^b(\operatorname{End}^*(\mathcal{T}^\#)-\text{mod}^{\text{gr}}) \simeq  \mathcal{D}^b(\operatorname{End}^*(\mathcal{T})^{\text{op}}-\text{mod}^{\text{gr}})$$ follows from (Theorem 8.2, \cite{keller}). By the statement (i) of Theorem \ref{quasiiso_ext_schurs} we obtain $$\mathcal{D}^b(\operatorname{End}^*(\mathcal{T})^{\text{op}}-\text{mod}^{\text{gr}})\simeq \mathcal{D}^b(\operatorname{Ext}^*(S,S)^{\text{op}}-\text{mod}^{\text{gr}})$$ (c.f. Example in Section 1.5, \cite{keller2}). Since $\mathcal{D}^b(\mathcal{C})\simeq \mathcal{D}^b(\mathcal{C}^{\text{op}})$ for a highest weight category $\mathcal{C}$ with duality (c.f. (1.5), \cite{cps_kazlus}), the following equivalence holds: $$\mathcal{D}^b(\operatorname{Ext}^*(S,S)^{\text{op}}-\text{mod}^{\text{gr}}) \simeq \mathcal{D}^b(\operatorname{Ext}^*(S,S)-\text{mod}^{\text{gr}}).$$Summarizing all the above equivalences, we obtain the assertion of the corollary.
\end{proof}
\subsection{The Yoneda algebra of simple functors}
The results of this subsection are similar to those in the previous subsection. Our main goal is to describe the structure of the Yoneda algebra of simple functors and to prove the DG formality of the endomorphism algebra whose cohomology algebra is the Yoneda algebra of simple functors.

Fix $0\leq i \leq p-1.$ We now define a double complex, which provides an injective resolution of the simple functor $F_i$. Set$$\mathcal{R}^{r,s}_i = \Omega^{i+s-r} \quad \text{if} \quad 0\leq r \leq p-1, \ 0\leq s \leq p-i-1 \text{ and } r-s \leq i .$$Let $(d^{r,s}_i)_v:\mathcal{R}^{r,s}_i \to \mathcal{R}^{r,s+1}_i$ (resp. $(d_i^{r,s})_h:\mathcal{R}^{r,s}_i \to \mathcal{R}^{r+1,s}_i$) be the de Rham differential $d_{i+s-r}$ (resp. the Koszul differential $\kappa_{i+s-r}$) for $r,s$ such that $\mathcal{R}^{r,s}_i,\mathcal{R}^{r,s+1}_i\neq 0$ (resp. $\mathcal{R}^{r,s}_i,\mathcal{R}^{r+1,s}_i\neq 0$). By (\ref{contractible_omegas}) it is easily seen that $(\mathcal{R}^{r,s}_i, (d_i^{r,s})_h, (d_i^{r,s})_v)_{r,s\in \mathbb{Z}}$ is a double complex. We denote it briefly by $\mathcal{R}_i.$ We note that in $\mathcal{R}_i$ the rows and columns are, respectively, the truncated Koszul and de Rham complexes. For example, the diagram below is $\mathcal{R}_{2}$ for $p=5.$
$$\begin{tikzcd}[cells={nodes={minimum height=2.5em}}]
\Omega^4 \arrow[r,"\kappa_4"] & \Omega^3 \arrow[r,"\kappa_3"] & \Omega^2 \arrow[r,"\kappa_2"] & \Omega^1 \arrow[r,"\kappa_1"] & \Omega^0 \\
\Omega^3 \arrow[r,"\kappa_3"] \arrow[u,"d_3"] & \Omega^2 \arrow[r,"\kappa_2"] \arrow[u,"d_2"] & \Omega^1 \arrow[r,"\kappa_1"] \arrow[u,"d_1"] & \Omega^0 \arrow[u,"d_0"] \\
\Omega^2 \arrow[r,"\kappa_2"] \arrow[u,"d_2"] & \Omega^1 \arrow[r,"\kappa_1"] \arrow[u,"d_1"] & \Omega^0 \arrow[u,"d_0"]
\end{tikzcd}$$
Let us consider the spectral sequence of $\mathcal{R}_i$ with respect to the vertical filtration. By the acyclicity of the Koszul complex we see that $$E^{s,t}_1=\begin{cases} K^{t+i}=S_{t+i} & \text{if } s=0 \text{ and } 0\leq t \leq p-i-1 \\ 0 & \text{otherwise} \end{cases}.$$ By (\ref{cartier_kernels}) we have $H^0(K^\bullet_p)=I^{(1)}=F_0$ and $H^j(K^\bullet_p)=0$ for $j>0$. Then by the formula for $E^{s,t}_1$ and (\ref{kernel_koszul_complex}) we obtain $$E^{s,t}_2=\begin{cases} F_i & \text{if } s=0, \ t=0 \\ 0 & \text{otherwise} \end{cases}.$$ In particular, this spectral sequence degenerates at $E_2$ and, in consequence, it easily follows that $$H^n(\operatorname{Tot}(\mathcal{R}_i))=\begin{cases} F_i & \text{if } n=0 \\ 0 & \text{otherwise}  \end{cases}.$$ Since each functor $\mathcal{R}^{r,s}_i$ is injective and a direct sum of injectives is injective, the last equality implies that $\operatorname{Tot}(\mathcal{R}_i)$ is an injective resolution of $F_i.$ By abuse of notation we use the same symbol $\mathcal{R}_i$ for $\operatorname{Tot}(\mathcal{R}_i)$.

Our next objective is to determine a basis of the Yoneda algebra of simple functors, whose elements are induced by the morphisms of total complexes from $\mathcal{R}_i$ to suitably shifted $\mathcal{R}_j$ defined as follows. For given  $0\leq i \leq j \leq p-1$ and $t$ such that $\operatorname{Ext}^t(F_i,F_j)\neq 0$ (c.f. Corollary \ref{ext_simples}) let $\alpha_{ji}^t$ be the map given by $(\alpha_{ji}^t)^{*,s}=\gamma_{s+\frac{1}{2}(t+i+j),i+s}$ for $0\leq s \leq p-1-\frac{1}{2}(t+j+i)$. Let us now assume that $0\leq j < i \leq p-1.$ Regarding the double complexes $\mathcal{R}_i$ and $\mathcal{R}_j$ as the complexes of vertical complexes, let $\alpha^{i-j}_{ji}$ be the natural inclusion and define $\alpha^t_{ji}=\alpha^{t-(i-j)}_{jj}\circ \alpha_{ji}^{i-j}$ for $t>i-j$ defined above. Less formal, if $i\leq j $ then $\alpha_{ji}^t$ embeds the $s$-th row of $\mathcal{R}_i$ into the $s+\frac{1}{2}(t-(j-i))$-th row of $\mathcal{R}_j,$ assuming the latter is non-zero. In case $i>j$ we compose $\alpha^{j-i}_{ji}$, which embeds $\mathcal{R}_i$ into $\mathcal{R}_j$, with $\alpha_{jj}^{t-(i-j)}.$ It is easy to check that the maps of total complexes induced by $\alpha_{ij}^t$ are morphisms of total complexes. We use for the simplicity the same symbol $\alpha_{ji}^t$ for that. We note that $\alpha^t_{ji}\in \operatorname{Hom}^t(\mathcal{R}_i,\mathcal{R}_j)$ as the morphism of complexes. Let us now show that these morphisms are not null-homotopic.
\begin{lem}
\label{notnull_simple}
The morphism of total complexes $\alpha^t_{ij}$ is not null-homotopic for any $0\leq i,j \leq p-1$ and $t$ such that $\operatorname{Ext}^t(F_i,F_j)\neq 0.$
\end{lem}
\begin{proof}
Fix $0\leq i,j \leq p-1.$ By Corollary \ref{ext_simples} the maximal $t$ such that $\operatorname{Ext}^t(F_i,F_j)\neq 0$ is $t=2p-i-j-2.$ We see at once that $\alpha^{2p-i-j-2}_{ji}=\alpha_{jj}^{2p-i-j-2-t}\circ\alpha_{ji}^t$ for any $t<2p-i-j-2$ satisfying $\operatorname{Ext}^t(F_i,F_j)\neq 0.$ Thus, if $\alpha^{2p-i-j-2}_{ji}$ is not null-homotopic then so is $\alpha^t_{ji}$. Hence it suffices to show that $\alpha^{2p-i-j-2}_{ji}$ is not null-homotopic.

We suppose that $\alpha^{2p-i-j-2}_{ji}$ is null-homotopic. Let $h_{ji}=h^*_{ji}$ be a homotopy between $\alpha^{2p-i-j-2}_{ji}$ and the zero map. Then there is the following diagram, where the top row is the total complex $\mathcal{R}_i$ and the bottom row is the total complex $\mathcal{R}_j[2p-i-j-2].$
$$\begin{tikzcd}[column sep=5em,row sep=4em]
 & \ldots \arrow[r] & \bigoplus_{r+s=i} \mathcal{R}_i^{r,s} \arrow[r,"d_h+d_v"] \arrow[ld,"h_{ji}^i"] \arrow[d,"(\alpha^{2p-i-j-2}_{ji})^i"] & \bigoplus_{r+s=i+1} \mathcal{R}_i^{r,s}  \arrow[r] \arrow[ld,"h_{ji}^{i+1}"] & \ldots \\
\ldots \arrow[r] & \Omega^1 \arrow[r,"d_h+d_v=\kappa_1"] & \Omega^0 \arrow[r] & \ldots &
\end{tikzcd}$$
By definition, $(\alpha^{2p-i-j-2}_{ji})^i=id_{\Omega^0}.$ By Lemma \ref{hom_schur_omega} and the fact that $\Omega^r\simeq (\Omega^r)^{\#}$ for $1\leq r \leq p-1$ we have $\operatorname{Hom}(\Omega^m,\Omega^0)=\operatorname{Hom}(\Omega^m,S_0)=0$ for $1<m\leq p-1$ and $\operatorname{Hom}(\Omega^1,\Omega^0)=\operatorname{Hom}(\Omega^1,S_0)=k.$ It implies that $h^{i+1}_{ji}=a\cdot\kappa_1$ for some $a \in k.$ Now we apply the exact functor $\operatorname{Hom}(-,\Omega^1)$ to the short exact sequence (\ref{koszul_sh_seq}) for a fixed $m>1$. Then by Lemma \ref{hom_schur_omega} we have $$\operatorname{Hom}(\Omega^m,\Omega^1)=\operatorname{Hom}(S_m,\Omega^1)+\operatorname{Hom}(S_{m-1},\Omega^1)=\begin{cases}1 & \text{if } m=2 \\ 0 & \text{if } 2 < m\leq p-1 \end{cases}.$$ We also have, again by Lemma \ref{hom_schur_omega}, that $\operatorname{Hom}(\Omega^0,\Omega^1)=\operatorname{Hom}(S_0,\Omega^1)=k$. Thus, we conclude that $h^i_{ji}=b\cdot\kappa_2+c\cdot d_1$ for some $b,c\in k.$ Since $h_{ji}$ is a homotopy between $\alpha_{ji}^{2p-i-j-2}$ and the zero map, the above observations give us the equality
\begin{displaymath}
\kappa_1\circ(b\cdot \kappa_2 +c\cdot d_0)+(a\cdot\kappa_1)\circ(\kappa_2+d_0)=id_{\Omega^0}, \quad \text{that is,} \quad (a+c)\cdot \kappa_1  d_0 =id_{\Omega^0}.
\end{displaymath}
On the other hand, $(a+c)\cdot \kappa_1 d_0 = 0$ by (\ref{contractible_omegas}), contrary to the equality in the preceding sentence. In particular, $\alpha^{2p-i-j-2}_{ji}$ is not null-homotopic.
\end{proof}
The class of the map $\alpha^t_{ji}$ in $ \operatorname{Ext}^t(F_i,F_j)$ will be denoted by $[\alpha^t_{ji}].$ Set $F=\bigoplus_{0\leq i \leq p-1} F_i$. By Corollary \ref{ext_simples} and Lemma \ref{notnull_simple} we have
\begin{equation}
\label{Ext_simples_generators}
\operatorname{Ext}^*(F,F)=\operatorname{span}\{[\alpha^t_{ji}]:0\leq i,j \leq p-1 \text{ and } \ t=|i-j|+2r, \text{ where } 0\leq r \leq p-\max\{i,j\}-1 \}
\end{equation}
as the graded vector space. It easily follows from the definition of the maps $\alpha^t_{ij}$ that
\begin{equation}
\label{Hom_simples_multiplication}
\alpha^t_{ml}\cdot \alpha^u_{ji}=\begin{cases} \alpha^{u+t}_{mi} & \text{if } j=l \text{ and } u+t \leq 2p-i-m-2 \\ 0 & \text{otherwise} \end{cases}.
\end{equation}
This equality implies the equality in the algebra $\operatorname{Ext}^*(F,F)$:
\begin{equation}
\label{Ext_simples_multiplication}
[\alpha^t_{ml}]\cdot [\alpha^u_{ji}]=\begin{cases} [\alpha^{u+t}_{mi}] & \text{if } j=l \text{ and } u+t \leq 2p-i-m-2 \\ 0 & \text{otherwise} \end{cases}.
\end{equation}
In the next theorem we use the last equality to describe the structure of the Yoneda algebra of simple functors $\operatorname{Ext}^*(F,F).$ We also prove the DG formality of the DG algebra $\operatorname{End}^*(\mathcal{R})$ with $\mathcal{R}=\bigoplus_{0\leq i \leq p-1} \mathcal{R}_i$ defined in the similar way as $\operatorname{End}^*(\mathcal{T})$.

Let $B=\bigoplus_{t\in \mathbb{N}} B_t$ be the graded vector space with the components $$B_t=\operatorname{span}\{b^t_{ji}:0\leq i,j \leq p-1 \text{ such that } t=|i-j|+2r, \ \text{where} \ 0\leq r \leq p-\max\{i,j\}-1 \}$$ for $0\leq t \leq 2p-2$ and $B_t=0$ for $t\geq 2p-1$, where $b^t_{ij}$ are the formal symbols for $i,j,t$ satisfying the conditions given above. We define the multiplication on $B$ as follows: $$ b^t_{lm}\cdot b^u_{ji} = \begin{cases} b^{t+u}_{mi} & \text{if } j=l \text{ and } u+t \leq 2p-i-m-2\\ 0 & \text{otherwise} \end{cases}.$$
\begin{theorem}[]~
\label{quaiiso_ext_simples}
\begin{enumerate}[label=(\roman*)]
\item There is a graded algebra isomorphism $\operatorname{Ext}^*(F,F)\simeq B.$
\item There is a graded algebra isomorphism $\operatorname{Ext}^*(F_i,F_i)\simeq K[x]/(x^{p-i})$ for $0\leq i \leq p-1$  and $x$ of degree $2$. In particular, $\operatorname{Ext}^*(F_i,F_i)$ is a commutative algebra.
\item The algebra $\operatorname{End}^*(\mathcal{R})$ is a formal DG algebra, i.e. there exists a quasi-isomorphism of DG algebras $\eta:\operatorname{Ext}^*(F,F)\to \operatorname{End}^*(\mathcal{R}),$ where $\operatorname{Ext}^*(F,F)$ is the DG algebra with zero differential.
\end{enumerate}
\end{theorem}
\begin{proof}
Let us first prove (i). Let $\eta:B\to \operatorname{Ext}^*(F,F)$ be the map given by $\eta(b^t_{ji})=[\alpha^t_{ji}].$ It follows from (\ref{Ext_simples_multiplication}) that $\eta$ is a graded algebra homomorphism. By (\ref{Ext_simples_generators}) $\eta$ maps a basis of $B$ onto a basis of $\operatorname{Ext}^*(F,F).$ Therefore $\eta$ is a graded algebra isomorphism, as required.

Now we show (ii). Fix $0\leq i \leq p-1.$ We see that the isomorphism $\eta$ defined above restricts to the graded algebra isomorphism $$\operatorname{Ext}^*(F_i,F_i)\simeq \operatorname{span}\{b_{ii}^{2r}: 0\leq r \leq p-i-1\}=:B'.$$ Let us consider $K[x]$ as the graded algebra with $x$ of degree $2.$ We define $f: K[x]\to B'$ to be the map given on the basis by
\begin{displaymath}
f(1)=b_{ii}^0, \qquad f(x^r)=\begin{cases} b_{ii}^{2r} & \text{if } 1\leq r \leq p-i-1 \\ 0 & \text{if } r > p-i-1 \end{cases}.
\end{displaymath}
By the definition of $B'$ it is easily seen that $b_{ii}^0$ is the identity of $B'$ and $b_{ii}^{2r}=(b_{ii}^2)^r$ for $1\leq r \leq p-i-1.$ Thus, $f$ is a graded algebra homomorphism. It is a simple matter to check that $\ker f =(x^{p-i}),$ hence, by the first isomorphism theorem, we obtain $K[x]/(x^{p-i}) \simeq B' \simeq \operatorname{Ext}^*(F_i,F_i)$ as graded algebras. The second assertion of (ii) is the immediate consequence of the first one.

It remains to prove (iii). The differential on $B$ is defined to be the zero differential. Let $\xi:B \to \operatorname{End}^*(\mathcal{R})$ be the map given by $\xi(b^t_{ji})=\alpha^t_{ji}$ for $0\leq i,j \leq p-1$ and $t=|i-j|+2r$ where $0\leq r \leq p-\max\{i,j\}-1.$ By (\ref{Hom_simples_multiplication}) we see that $\xi$ is a graded algebra homomorphism and $\xi$ is a chain map, since $\xi(b^t_{ji})$ are morphisms of complexes. In particular, $\xi$ is a DG homomorphism. It is sufficient to show that $\xi$ is a quasi-isomorphism. By the definition of the differential on $B$ we have $H^*(B) \simeq B.$ We see at once that the induced homomorphism $\xi_*:H^*(B)\simeq B\to H^*(\operatorname{End}^*(\mathcal{R}))\simeq \operatorname{Ext}^*(F,F)$ is given by $\xi_*(b^t_{ji})=[\alpha^t_{ji}],$ i.e. $\xi_*=\eta,$ where $\eta$ was defined above. We showed that $\eta$ is a graded algebra isomorphism. Hence $\xi:B\simeq \operatorname{Ext}^*(F,F)\to \operatorname{End}^*(\mathcal{R})$ is a quasi-isomorphism.
\end{proof}
\begin{cor}
\label{derivedequiv_simples}
There is an equivalence of triangulated categories $\mathcal{D}^b\mathcal{P}_p^{\varnothing} \simeq \mathcal{D}^b(\operatorname{Ext}^*(F,F)-\text{mod}^{\text{gr}}).$
\end{cor}
\begin{proof}
Let $\mathcal{R}^{\#}$ be the resolution obtained from $\mathcal{R}$ by using the Kuhn duality. Clearly $\mathcal{R}^{\#}$ is a cofibrant object in the category of complexes of $\mathcal{P}_p^{\varnothing}$. Since $\mathcal{R}^{\#}_i\simeq F_i^{\#}\simeq F_i$ in $\mathcal{D}^b \mathcal{P}^{\varnothing}_p$ for $0\leq i \leq p-1,$ we conclude, by induction on the length of the composition series, that $\mathcal{R}^{\#}$ is a small generator of $\mathcal{D}^b\mathcal{P}^{\varnothing}_p.$ Then the proof is analogous to that of Corollary \ref{derivedequiv_schurs}.
\end{proof}
We also obtain the following corollary, which is by no means obvious from the explicit descriptions of the involved graded algebras.
\begin{cor}[]~\\
\label{derivedequiv_schursimple}
There is an equivalence of triangulated categories $\mathcal{D}^b(\operatorname{Ext}^*(S,S)-\text{mod}^{\text{gr}}) \simeq \mathcal{D}^b(\operatorname{Ext}^*(F,F)-\text{mod}^{\text{gr}}).$
\end{cor}
\begin{proof}
The assertion of the corollary follows immediately from Corollaries \ref{derivedequiv_schurs} and \ref{derivedequiv_simples}.
\end{proof}
\section{The blocks of p-weight $1$}
Now we consider the case of the category $\mathcal{P}_e$ for $e>p$. Let $\mathcal{P}_{|\lambda|+p}^{\lambda}$ be the block of $\mathcal{P}_{|\lambda|+p}$ corresponding to a $p$-core $\lambda$. Our objective is to prove the following main theorem in this section. The principal significance of this theorem is that it provides a generalization of the results obtained in the previous sections to the case of $\mathcal{P}_e$ for $p < e < 2p.$
\begin{theorem}
\label{equivcat}
Assume that $e>p$ and fix a $p$-core $\lambda$. Let $\theta:\mathcal{P}_p^\varnothing \to \mathcal{P}_{|\lambda|+p}^{\lambda}$ be the functor given by $\theta(G) = \pi(S_\lambda \otimes G)$, where $\pi$ is the natural projection onto $\mathcal{P}_{|\lambda|+p}^{\lambda}$. Then $\theta$ is an equivalence of abelian categories. 
\end{theorem}
We first prove the following lemma, which will be useful in the proof of Theorem \ref{equivcat}.
\begin{lem}
\label{comb_lemma}
For any $0\leq i \leq p-1$ there is exactly one Young diagram $\mu_i$ with $p$-core $\lambda$ such that $\lambda$ is obtained from $\mu_i$ by removing the rim $p$-hook corresponding to the $p$-hook $(i+1,1^{p-i-1})$. Moreover, the map $(i+1,1^{p-i-1}) \mapsto \mu_i$ is an order isomorphism between the poset of $p$-hooks and the poset of Young diagrams of $p$-weight $1$ with $p$-core $\lambda$.
\end{lem}
\begin{proof}
We will use in the proof the terminology and results presented in (Section 2.7, \cite{james_kerber}). We define the $\beta$-sequence of the Young diagram $\mu=(\mu_1,\ldots,\mu_j,0,\ldots)$ by $\beta_l=\mu_l-l$ for $l\geq 1.$ We will often consider the bead configuration of a given Young diagram. We recall that sliding a bead one space up is equivalent to removing a rim $p$-hook. Since a Young diagram is uniquely determined by its $p$-core and $p$-quotient, it is clear that each diagram of $p$-weight $1$ is determined by sliding a bead one space down in the bead configuration of its $p$-core. If it is possible to remove a rim $p$-hook in a given diagram $\mu$ and $(\beta_j)$ is its $\beta$-sequence then $(\beta_1,\ldots,\beta_{l-1},\beta_{l+1},\ldots, \beta_{l+p-i-1},\beta_l-p,\beta_{l+p-i},\ldots)$ is the $\beta$-sequence of the diagram, which arises from $\mu$ by removing a rim $p$-hook corresponding to the hook $(i+1,1^{p-i-1})$ whose hand is in the $l$-th row. In particular, in this case there are exactly $p-i-1$ beads in the bead configuration of $\mu$ between the bead, which is slid one space up to remove the given rim $p$-hook, and the new position of that.

Fix $0\leq i \leq p-1.$ Now we claim that there is exactly one Young diagram $\mu_i$ of $p$-weight $1$ and with $p$-core $\lambda$ such that $\lambda$ arises from $\mu_i$ by removing the rim $p$-hook corresponding to $(i+1,1^{p-i-1})$. Indeed, let us consider the bead configuration of $\lambda.$ We observe that there are $p-i-1$ beads between the $(i+1)$-th bead with the empty space under it and this empty space, since each empty space between them corresponds to a bead with an empty space under it being before the bead under consideration and there are exactly $i$ such beads. Then the claim follows from the remark on $\beta$-sequences in the previous paragraph. Obviously, if $i \neq j$ then $\mu_i \neq \mu_j$. Thus, the map $(i+1,1^{p-i-1}) \mapsto \mu_i$ is a bijection.

It remains to prove that the map $(i+1,1^{p-i-1}) \mapsto \mu_i$ preserves the reversed dominance order. It suffices to show that $\mu_i$ dominates over $\mu_{i+1}$ for any $0\leq i \leq p-2.$ Let $l$ (resp. $m$) be the number of the row of $\mu_{i+1}$ (resp. $\mu_i$) containing the hand of the rim $p$-hook, which has to be removed to obtain $\lambda$. It follows from the construction of the map $(i+1,1^{p-i-1})\mapsto \mu_i$ that $l\leq m.$ Let $(\beta_j)$ (resp. $(\beta'_j)$) be the $\beta$-sequence of $\mu_{i+1}$ (resp. $\mu_i).$ It follows from the definition of $\beta$-sequence that $\mu_i$ dominates over $\mu_{i+1}$ if and only if $\sum_{r\leq s}\beta'_r \leq \sum_{r\leq s} \beta_r$ for any $s\geq 1.$ We see that $(\beta_1,\ldots,\beta_{l-1},\beta_{l+1},\ldots,\beta_{l+p-i-1},\beta_l-p,\beta_{l+p-i},\ldots)$ and $(\beta'_1,\ldots,\beta'_{m-1},\beta'_{m+1},\ldots,\beta'_{m+p-i}, \beta'_m-p,\beta'_{m+p-i+1},\ldots)$ are the both the $\beta$-sequence of $\lambda$ obtained by removing the rim $p$-hook in, respectively, $\mu_{i+1}$  and $\mu_i.$ Hence we have
\begin{equation*}
\begin{split}
& \beta_r=\beta'_r \ \text{ for } 1\leq r \leq l-1 \text{ and } m+1 \leq r \leq l+p-i-1 \text{ and } r\geq m+p-i+1,\\
& \beta_l=\beta'_{l+p-i}+p, \qquad \quad \beta_r=\beta'_{r-1} \ \text{ for } l+1\leq r \leq m, \\
& \beta_{m+p-i}=\beta'_m-p, \qquad \beta_r=\beta'_{r+1} \ \text{ for } l+p-i \leq r \leq m+p-i-1.
\end{split}
\end{equation*}
Then it is a simple matter to check that $\sum_{r\leq s}\beta'_r \leq \sum_{r\leq s} \beta_r$ for any $s\geq 1$. Thus, $\mu_i$ dominates over $\mu_{i+1}$, and the proof is complete.
\end{proof}
\begin{proof}[Proof of Theorem \ref{equivcat}]
It is clear that $\theta$ is an exact and additive functor. Now we show that $\theta(S_i)=S_{\mu_i}$ and $\theta(W_i)=W_{\mu_i}$ for any $0\leq i \leq p-1$, where $\mu_i$ is the Young diagram as in the proof of Lemma \ref{comb_lemma}. Fix $0\leq i \leq p-1.$ By the Littlewood-Richardson rule (c.f. \cite{boffi}) we obtain 
\begin{displaymath}
\label{littlewood_equiv}
\theta(S_i)=\pi(S_\lambda \otimes S_i)=\bigoplus_{\mu \in P_{\lambda,1}} c(i,\lambda ; \mu)S_\mu
\end{displaymath}
up to filtration, where $P_{\lambda,1}$ is the set of Young diagrams of $p$-weight $1$ with $p$-core $\lambda$ and $c(i,\lambda ; \mu)$ is the number of semistandard skew tableaux of shape $\mu / \lambda$ and of the content $(p-i,1^i)$ such that the word obtained by reading each column from the bottom up, starting from the left-most column and moving to the right column by column, is a word of Yamanouchi. We observe that a skew tableau as in the definition of $c(i,\lambda ; \mu)$ is of shape of a rim $p$-hook corresponding to the hook $(i+1,1^{p-i-1})$ and there is only one filling such that the corresponding word is a word of Yamanouchi. Indeed, it is obvious that $\mu / \lambda$ is a rim $p$-hook and it is easily seen that if the $j$-th column of such a skew tableau has only one box then this column is filled by the number $j$ and, otherwise, this column is filled by the sequence $(j,1,\ldots,1)$, reading from the bottom up. The picture below provides an example of such a skew tableau for $p=7, \ \mu/\lambda=(4,3,3,1)/(2,2)$.$$
\ytableausetup{notabloids}
\begin{ytableau}
\none & \none & 1 & 4 \\ \none & \none & 1 \\ 1 & 2 & 3 \\ 1
\end{ytableau}$$
Hence if $\lambda$ arises from $\mu$ by removing the rim $p$-hook corresponding to the $p$-hook $(i+1,1^{p-i-1})$ then $c(i, \lambda; \mu)=1$ and, otherwise, $c(i,\lambda; \mu)=0$. By Lemma \ref{comb_lemma} there is exactly one Young diagram $\mu=\mu_i$ such that $c(i, \lambda; \mu)=1$. Hence $\theta(S_i)=S_{\mu_i}$. Clearly $S_\lambda \simeq W_\lambda$, since $\lambda$ is a $p$-core. Then we have $$\theta(W_i)=S_\lambda\otimes W_i \simeq W_\lambda\otimes W_i=(S_i \otimes S_\lambda)^\#=S_{\mu_i}^\#=W_{\mu_i}.$$ The assertion of the theorem follows from (Theorem (5.8), \cite{ps_ottawa}).
\end{proof}
By the last theorem it is clear that we obtain $\operatorname{Ext}$-computations from those in Section $2$ by replacing $m$ by $\mu_m$ and $n$ by $\mu_n$. We record in the following corollary the another important and immediate consequences of Theorem \ref{equivcat} and the results of the previous sections. The last statement of the corollary follows from the easy observation that each diagram of weight $p < e < 2p$ is of $p$-weight at most $1.$
\begin{cor}
Let $\lambda$ be a $p$-core.
\begin{enumerate}[label=(\roman*)]
\item The decomposition matrix $D=(d_{lm})$ of the category $\mathcal{P}^\lambda_{|\lambda|+p}$ is given by
\begin{displaymath}
d_{lm}:=[W_{\mu_{l-1}}:F_{\mu_{m-1}}]=\begin{cases} 1 & \text{if } m=l \ \text{ or } m=l+1 \\ 0 & \text{otherwise} \end{cases}.
\end{displaymath} 
\item The category $\mathcal{P}^\lambda_{|\lambda|+p}$ has a Kazhdan-Lusztig theory. 
\item The endomorphism algebra whose cohomology algebra is the Yoneda algebra of Schur functors in $\mathcal{P}^\lambda_{|\lambda|+p}$, $\operatorname{Ext}^*(S,S)$ where $S=\bigoplus_{0\leq i \leq p-1} S_{\mu_i}$, is DG formal.
\item The endomorphism algebra whose cohomology algebra is the Yoneda algebra of simple functors in $\mathcal{P}^\lambda_{|\lambda|+p}$, $\operatorname{Ext}^*(F,F)$ where $F=\bigoplus_{0\leq i \leq p-1} F_{\mu_i}$, is DG formal.
\item There are the following equivalences of triangulated categories: $\mathcal{D}^b \mathcal{P}^\lambda_{|\lambda|+p} \simeq \mathcal{D}^b(\operatorname{Ext}^*(S,S)-\text{mod}^{\text{gr}})$, $ \mathcal{D}^b \mathcal{P}^\lambda_{|\lambda|+p} \simeq  \mathcal{D}^b(\operatorname{Ext}^*(F,F)-\text{mod}^{\text{gr}}), \  \mathcal{D}^b(\operatorname{Ext}^*(S,S)-\text{mod}^{\text{gr}}) \simeq \mathcal{D}^b(\operatorname{Ext}^*(F,F)-\text{mod}^{\text{gr}}).$
\item The above statements hold for any block of non-zero $p$-weight of the category $\mathcal{P}_e$ for $p < e < 2p.$ 
\end{enumerate}
\end{cor}
It was showed in \cite{hty} that if two blocks of $\mathcal{P}$ have the same $p$-weight then they are derived equivalent (c.f. Theorem 5, \cite{hty}). Theorem \ref{equivcat} is a refinement of this result in the case of $p$-weight $1$. It turns out that the approach used in the proof of this theorem does not work in the case of $p$-weight greater than $1$.
\begin{exmp}
Let $p=2.$ Then $\mathcal{P}_4^\varnothing$ and $\mathcal{P}_5^{(1)}$ are blocks of $p$-weight $2.$ Let us consider the functor $\theta:\mathcal{P}^\varnothing_4 \to \mathcal{P}_5^{(1)}$ given by $\theta(G)=\pi(S_{(1)}\otimes G)$, where $\pi$ is the natural projection onto $\mathcal{P}_5^{(1)}$. By the Littlewood-Richardson rule we easily obtain $\theta(S_{(2,2)})=S_{(3,2)}\oplus S_{(2,2,1)}$ up to filtration. In particular, the image of $S_{(2,2)}$ under $\theta$ is not a Schur functor.
\end{exmp}
\section*{Acknowledgments}
I wish to express my gratitude to Marcin Chałupnik for suggesting the problem, many stimulating conversations and carefully reading of the preliminary version of this article. I also gratefully thank the referee for his/her valuable remarks, which improved this paper.

\textsc{Patryk Jaśniewski, Institute of Mathematics, Warsaw University, Banacha 2, 02-097 Warsaw, Poland}

\textit{E-mail address}: \texttt{p.jasniewski@mimuw.edu.pl}
\end{document}